\documentclass[a4paper, 10pt, conference]{ieeeconf}      

\usepackage[utf8x]{inputenc} 
\usepackage{hyperref}       
\usepackage{url}            
\usepackage{booktabs}       
\usepackage{nicefrac}       
\usepackage{microtype}

\usepackage{amsmath,amsfonts,amssymb,amsthm,mathtools}
\usepackage{subfigure}

\usepackage{cmap}					
\usepackage{mathtext} 		
\usepackage{wrapfig}
\usepackage{graphicx}

\usepackage{mathabx}
\usepackage{algorithm}
\usepackage{algorithmic}
\usepackage{array}
\usepackage{mdwmath}
\usepackage{mdwtab}
\usepackage{eqparbox}

\usepackage{comment}

\usepackage[percent]{overpic}

\usepackage{cite}

\newtheorem{theorem}{Theorem}
\newtheorem{lemma}[theorem]{Lemma}

\usepackage{algorithm}
\usepackage{algorithmic}

\renewcommand{\leq}{\leqslant}
\renewcommand{\geq}{\geqslant}
\newcommand{\argmin}{\operatornamewithlimits{argmin}}
\renewcommand{\min}{\operatornamewithlimits{min}}
\newcommand{\argmax}{\operatornamewithlimits{argmax}}
\def\eps{\varepsilon}
\newcommand{\mc}[1]{\mathbb{\item1}}

\def\la{\langle}
\def\ra{\rangle}
\def\R{\mathbb{R}}

\DeclareMathOperator{\spn}{span}

\DeclareMathOperator{\vectr}{vec}

\def\one{\mathbf{1}}
\def\vp{\varphi}
\def\e{\varepsilon}

\usepackage{xcolor}

\IEEEoverridecommandlockouts                              

\title{\LARGE \bf
Multimarginal Optimal Transport by \\Accelerated Alternating Minimization
}

\author{Nazarii Tupitsa, Pavel Dvurechensky, Alexander Gasnikov, C\'esar A. Uribe
	\thanks{The work of P. Dvurechensky, A. Gasnikov, and N. Tupitsa in part III-A -- III-C, IV was funded by Russian Science Foundation (project 18-71-10108). The work of C.A. Uribe and A. Gasnikov in part III-D was partially funded by the Yahoo! Faculty Engagement Program.
	The work of P. Dvurechensky in part II was funded by the Deutsche Forschungsgemeinschaft (DFG, German Research Foundation) under Germany's Excellence Strategy – The Berlin Mathematics Research Center MATH+ (EXC-2046/1, project ID: 390685689). The research of N. Tupitsa in part V was supported by the Ministry of Science and Higher Education of the Russian Federation (Goszadaniye) No. 075-00337-20-03, project No. 0714-2020-0005.}
	\thanks{N.T. is with the Moscow Institute of Physics and Technology, Institute for Information Transmission Problems and National Research University Higher School of Economics, Russia (\textit{tupitsa @phystech.edu}).
	P.D. is with the Weierstrass Institute for Applied Analysis and Stochastics, Germany, and the Institute for Information Transmission Problems, Russia 	(\textit{pavel.dvurechensky @wias-berlin.de}). 
	A.G.  is with the Moscow Institute of Physics and Technology, Institute for Information Transmission Problems, Russia, and National Research University Higher School of Economics, Russia, and
	Sirius University of Science and Technology, Russia, and
    Caucasus Mathematical Center, Adyghe State University, Russia.
	(\textit{gasnikov@yandex.ru}).
	C.A.U. is with the Laboratory for Information and Decision Systems (LIDS), Massachusetts Institute of Technology, USA (\textit{cauribe@mit.edu}). }%
}

\begin{document}

\maketitle

\begin{abstract}
We study multimarginal optimal transport (MOT) problems, which include, as a particular case, the Wasserstein barycenter problem. In MOT problems, one has to find an optimal coupling between $m$ probability measures, which amounts to finding a tensor of order $m$. We propose a method based on accelerated alternating minimization and estimate the complexity to find an approximate solution. We use entropic regularization with a sufficiently small regularization parameter and apply accelerated alternating minimization to the dual problem. A novel primal-dual analysis is used to reconstruct the approximately optimal coupling tensor. Our algorithm exhibits a better computational complexity than the state-of-the-art methods for some regimes of the problem parameters.

\end{abstract}

\section{Introduction}\label{sec:intro}

Optimal transport (OT) has gained increasing interest in recent years from its broad range of applications ranging from medical image processing~\cite{gramfort2015fast}, machine learning~\cite{arjovsky2017wasserstein}, graph-theory~\cite{asoodeh2018curvature},  control theory~\cite{chen2016optimal}, among many others. Fundamentally, many of these applications require the comparison and quantification of distances between probability distributions~\cite{elvander2019multi}. In Kantorovich formulation, the OT problem seeks to minimize
\begin{align*}
    \int_{M_1 \times \cdots \times M_m} c(x_1,\cdots,x_m)d\pi(x_1,\cdots,x_m),
\end{align*}
over the set $\Pi(p_1. \cdots, p_m)$ of positive joint measures $\pi$ on the product space $M_1 \times \cdots M_m$ whose marginals are the $p_k$'s, where  $p_1. \cdots, p_m$ (marginals) is a set of probability measures on smooth manifolds $M_1,\dots,M_m$, and $c(x_1,\cdots,x_m)$ is a cost function~\cite{pass2015multi}.

Although the OT problem formulation is mathematically precise, see, for example, the seminal monograph by Villani~\cite{Villani-2003-Topic}, and references therein, its translation to practical applications heavily depends on the availability of computationally attractive methods. Many of the OT related problems are computationally intense, and much effort has been put into analyzing the underlying complexity of such problems~\cite{dvurechensky2018computational,kroshnin2019complexity,lin2019efficient,jambulapati2019direct}.

\vspace{-0.7cm}
\begin{figure}[ht]
	\centering
	\begin{overpic}[width=0.4\textwidth]{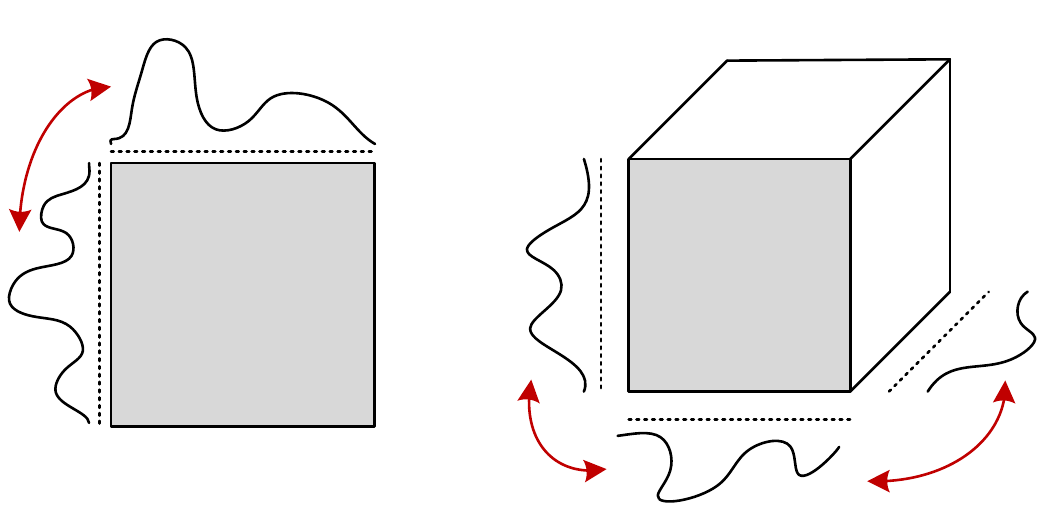}
		\put(-3,17){{\footnotesize $p_1$}}
		\put(20,45){{\footnotesize $p_2$}}
		\put(15,23){{\footnotesize Transport}}
		\put(20,18){{\footnotesize Plan}}
		\put(17,12){{\footnotesize $m=2$}}
		\put(63,23){{\footnotesize Transport}}
		\put(66,18){{\footnotesize Plan}}
		\put(64,13){{\footnotesize $m=3$}}
		\put(45,17){{\footnotesize $p_1$}}
		\put(68,-1){{\footnotesize $p_2$}}
		\put(101,18){{\footnotesize $p_3$}}
	\end{overpic} 
	\caption{A visual representation of the multimarginal optimal transport problem for $m=2$ and $m=3$. When $m=2$, the transport plan defines the optimal cost of moving $p_1$ into $p_2$. For discrete distributions this corresponds to a matrix with marginals $p_1$ and $p_2$. When $m=3$, in the discrete case is, the transport plan is a three dimensional tensor, whose marginals are $p_1$, $p_2$, and $p_3$.}
	\label{fig:transport}
\end{figure}
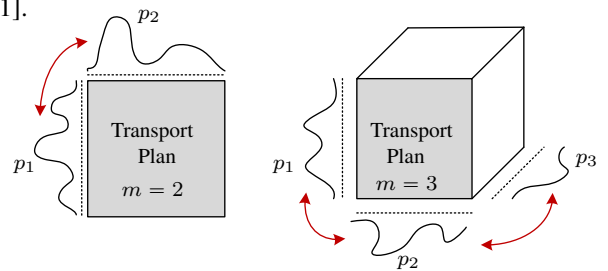

\vspace{-0.4cm}
Classically, OT has been studied for quantifying distances between \textit{two} probability distributions (i.e., $m=2$) for which theory is fairly well understood~\cite{Villani-2003-Topic,ambrosio2013user,mccann2012glimpse}. However, for $m\geq 3$, i.e., the multimarginal optimal transport (MOT) problem, much less is known, even though such regime has been recently shown useful for many applications, like tomographic image reconstruction~\cite{abraham2017tomographic}, generative adversarial networks~\cite{cao2019multi}, economics~\cite{ekeland2005optimal}, and density functional theory~\cite{parr1980density}. Figure~\ref{fig:transport} shows a visual representation of the MOT problem for $m=3$. See~\cite{pass2015multi} for a recent survey of fundamental theoretical formulations and applications of the MOT problem. 

Computational aspects of the MOT problem were studied in \cite{benamou2015iterative}, where an Iterative Bregman Projections algorithm was proposed for this problem, yet without complexity analysis. It was also pointed out that the MOT problem can be applied to calculate the barycenter of $m$ measures without fixing the barycenter's support. In \cite{2019arXiv191000152L}, the authors propose and analyze the complexity of two algorithms for the MOT problem. We follow~\cite{2019arXiv191000152L} by using the entropy regularization approach as well~\cite{cuturi2013sinkhorn}.

In this paper, we develop an algorithm for the computation of approximate solutions for the MOT problem using recently developed methods of alternating minimization. Our contributions are three-fold:
\begin{itemize}
    \item We develop a novel algorithm for the approximate computation of MOT maps based on accelerated alternating minimization algorithm.
    \item We formally prove the computational complexity of the proposed algorithm. We show that the proposed algorithm has an iteration complexity {$\widetilde{O} \left({m^2n^{1/2}}/{\varepsilon} \right)$}, and a computational complexity of {$\widetilde{O} \left({m^3n^{m+1/2}}/{\varepsilon} \right)$} arithmetic operations. Our result indicates an upper exponential bound for the Wasserstein barycenter problem's complexity with free support, which is known to be a non-convex optimization problem. 
    \item We show that in some regimes of the MOT problem parameters $m$ (number of distributions), $n$ (dimension of the distributions), and $\varepsilon$ (desired accuracy), the proposed algorithm has better iteration complexity in comparison with existing methods.
\end{itemize}

This paper is organized as follows. Section~\ref{sec:problem} presents the problem formulation and the dual aspects of the OT problem.  Section~\ref{sec:algo_design} contains the algorithm design methodology and the theoretical primal-dual analysis required for the establishment of the algorithmic complexity. Section~\ref{sec:experiments} shows some preliminary experiments. Section~\ref{sec:complexity} discusses the specific computational complexity results. Finally. {Section~\ref{sec:conclusions} presents the conclusions and future work.}

\section{The Entropy Regularized MOT Problem}\label{sec:problem}

In what follows, $\Delta^n$ denotes the probability simplex in $\R^n_+$:  $\Delta^n = \{u \in \R^n_+: \one_n^\top u = 1\}$. For a tensor $A = (A_{i_1, \ldots, i_m}) \in \R^{n_1 \times \ldots \times n_m}$, we write $\|A\|_\infty = \max_{1 \leq i_k \leq n_j, \forall k \in \{1, \dots, m\}} |A_{i_1, \ldots, i_m}|$ and $\|A\|_1 = \sum_{1 \leq i_k \leq n_j, \forall k \in \{1, \dots, m\}} |A_{i_1, \ldots, i_m}|$, and denote by $p_k(A) \in \R^{n_k}$ its $k$-th marginal for $k \in \{1, \dots, m\}$ where each component is defined as 
\begin{equation*}
[p_k(A)]_j \ = \  \sum_{1 \leq i_l \leq n_l, \forall l \neq k} A_{i_1, \ldots, i_{k-1}, j, i_{k+1}, \ldots, i_m}.
\end{equation*}
For two tensors of the same dimension, we denote the Frobenius inner product of $A$ and $B$ by 
\begin{equation*}
\left\langle A, B\right\rangle \ = \  \sum_{1 \leq i_k \leq n_k, \forall k \in \{1, \dots, m\}} A_{i_1, \ldots, i_m} B_{i_1, \ldots, i_m}.
\end{equation*}
 
The MOT problem between $m \geq 2$ discrete probability distributions with $n$ support points\footnote{For simplicity we consider same cardinality of the support set for each distribution. This can be extended for general case.} has the following form:
\begin{eqnarray}\label{prob:multi_OT}
    \min_{\substack{X \in \R_+^{n \times \ldots \times n}, \;\; 
          p_k(X) = p_k, \ \forall k \in \{1, \dots, m\}}}
        \left\langle C, X\right\rangle, 
\end{eqnarray}
where $X$ denotes a multimarginal transportation plan and $C \in \R_+^{n \times \ldots \times n}$ is a cost tensor. For all $k \in \{1, \dots, m\}$, a vector $p_k = (p_{kj})$ is given as a probability vector in $\Delta^n$.

The MOT problem is a linear program with $mn$ equality constraints, and $n^m$ variables and inequality constraints. When $m = 2$, the MOT problem reduces to the classical OT problem~\cite{Villani-2003-Topic}. 

In the general case of $m$ measures, one of the applications of MOT is grid-free Wasserstein barycenter computation \cite{benamou2015iterative}. Despite the linear programming (LP) formulation being in its standard form, the problem's dimension, which is exponential in $m$, does not allow the use of standard LP solvers such as interior-point methods \cite{nesterov1994interior,Lee-2014-Path}. Next, we describe how to apply the entropic regularization approach so ameliorate such computational requirements.

Following \cite{cuturi2013sinkhorn,benamou2015iterative}
, we consider a regularized version of~\eqref{prob:multi_OT}, in which we add an entropic penalty to the multimarginal transportation plan. The resulting problem has the following form:
\begin{equation}
\label{eq:EMOT_Pr_St}
    \min_{
        \substack{X \in \R_+^{n \times \ldots \times n},
        \\
        \;p_k(X) = p_k, \quad \forall k \in \{1, \dots, m\} 
                    \\ \sum_{i_1, \dots, i_m} X_{i_1,\dots, i_m} = 1,\quad 1 \leq i_j \leq n 
                 }}  F(X):=\left\langle C, X\right\rangle - \gamma H(X), 
\end{equation}
where $\gamma > 0$ is the regularization parameter, and $H(X)$ is the entropic regularization term:
$
H(X) \ : = \ - \left\langle X, \log(X)\right\rangle.
$
Here logarithm of a tensor should be understood as component-wise. We underline that we add a constraint that $X$ belongs to probability simplex of the size $n^m$. This constraint is a corollary of the fact that all the vectors $p_k$, $k=1,...,m$ belong to $\Delta^n$. Adding this constraint does not change the problem's solution, but it is crucial to obtain a dual optimization problem to have a Lipschitz-continuous gradient. The reason for the latter is that entropy is strongly convex on the probability simplex w.r.t. the $1$-norm.

The next lemma shows that the entropy regularized MOT problem has a closed-form dual representation that we can exploit for developing computationally efficient approaches.

\begin{lemma}
The dual problem formulation of the entropy regularized MOT problem~\eqref{eq:EMOT_Pr_St} is defined as $\max_\Lambda \phi(\Lambda)$, where 
\begin{multline}
\label{eq:dual_obj_lambda}
    \phi(\Lambda) :=
    \\
     -\gamma\Bigg[
    \ln\sum_{\substack{i_1, \dots, i_m \\ 1 \leq i_j \leq n \\  1 \leq j \leq m }} \exp\left\{-\sum_{k=1}^m \frac{[\lambda_k]_{i_k}}{\gamma} - \frac{C_{i_1 \ldots i_m }}{\gamma}-1  \right\}
    \\
    + 1 + \frac{1}{\gamma} \sum_{k = 1}^m \lambda_{k}^Tp_k \Bigg].
\end{multline}
Moreover, the primal variable can computed as
\begin{multline}
    \label{primal-var}
    X_{i_1 \ldots i_m}(\Lambda) 
    = \frac{
        \exp{\left(-\sum\limits_{k=1}^m \frac{{[\lambda_k]_{i_k}}}{\gamma} - \frac{C_{i_1 \ldots i_m }}{\gamma} \right)}}
        {\sum\limits_{\substack{i_1, \dots, i_m \\ 1 \leq i_j \leq n \\  1 \leq j \leq m }} \exp\left\{-\sum\limits_{k=1}^m \frac{{[\lambda_k]_{i_k}}}{\gamma} - \frac{C_{i_1 \ldots i_m }}{\gamma}  \right\}}
\end{multline}
Finally, with the change of variable $u_k=-\frac{\lambda_k}{\gamma} - \frac{\one}{m}$ the dual problem becomes
\begin{equation}
    \label{mot-dual}
    \min_U \ \ \phi(U) \equiv \phi(u_1,\ldots,u_m). 
\end{equation}
where $U = (u_1^T, \ldots, u_m^T)^T \in \R^{mn}$.

\end{lemma}

All the proofs of this paper can be found in~\cite{tupitsa2020multimarginal}.

\begin{proof}

We introduce dual variables $\lambda_i \in \R^n$ for $i \in \{1,\dots, m\}$ and define the Lagrangian function as follows:

\begin{multline}
    L(X, \Lambda, \mu) = \left\langle C, X\right\rangle + \gamma \left\langle X, \log(X)\right\rangle
    \\
    + \sum_{k=1}^m \lambda_k^T(p_k(X) - p_k) 
    + \mu \sum_{\substack{i_1, \dots, i_m \\ 1 \leq i_j \leq n \\  1 \leq j \leq m }} X_{i_1\ldots i_m} - \mu,
\end{multline}
where $\Lambda = (\lambda_1^T, \ldots, \lambda_m^T)^T \in \R^{mn}$,
and formulate the dual  unconstrained problem 
\[
    \max_{\Lambda \in \R^{mn}}\max_{\mu\in\R}  \min_{X\in\R_+^{n^m}} L(X, \Lambda, \mu) .
\]

Taking the derivative with respect to $X_{i_1 \ldots i_m}$ and setting it to zero yields
\begin{multline}
    \frac{\partial L}{\partial X_{i_1 \ldots i_m}}(X, \Lambda, \mu) = C_{i_1 \ldots i_m} +\gamma + \gamma \log(X_{i_1 \ldots i_m}) 
    \\
    + \sum_{k = 1}^m [\lambda_k]_{ i_k} + \mu = 0. 
\end{multline}
the solution of the above problem is
\[
    X_{i_1 \ldots i_m}(\Lambda,\mu) = \exp{
    \left(\frac{-\sum_{k=1}^m [\lambda_k]_{i_k} - C_{i_1 \ldots i_m} -\gamma - \mu }{\gamma}\right)}.
\]

Therefore, we have
\[
    L(\Lambda,\mu) = -\gamma \sum_{\substack{i_1, \dots, i_m \\ 1 \leq i_j \leq n \\  1 \leq j \leq m }} X_{i_1,\dots, i_m}(\Lambda, \mu) - \sum_{k=1}^m \lambda_k^T p_k - \mu.
\]

By taking a derivative w.r.t $\mu$  and setting it to zero we have
\[
    \sum_{\substack{i_1, \dots, i_m \\ 1 \leq i_j \leq n \\  1 \leq j \leq m }} X_{i_1,\dots, i_m}(\Lambda, \mu(\Lambda)) - 1 =0.
\]
From where we can express $\mu(\Lambda)$ as
\[
    \exp\left\{ -\frac{\mu}{\gamma}\right\}\sum_{\substack{i_1, \dots, i_m \\ 1 \leq i_j \leq n \\  1 \leq j \leq m }} \exp\left\{-\sum_{k=1}^m \frac{\lambda_{ki_k}}{\gamma} - \frac{C_{i_1 \ldots i_m }}{\gamma}-1  \right\} =1,
\]
yielding the theorem's statements.

As it is known \cite{nesterov2005smooth}, the objective in \eqref{eq:dual_obj_lambda} has Lipschitz continuous gradient. This follows from the fact that entropy is strongly convex on the probability simplex. Since the dual objective has Lipschitz gradient, we can use gradient-type of methods to solve the dual problem and obtain the corresponding complexity.

Finally, with the change of variable $u_k=-\frac{\lambda_k}{\gamma} - \frac{\one}{m}$ the dual objective becomes
\begin{multline}
    \phi(U) \equiv \phi(u_1,\ldots,u_m)
    \equiv 
    \\
    \gamma\Bigg[
    \ln\sum_{\substack{i_1, \dots, i_m \\ 1 \leq i_j \leq n \\  1 \leq j \leq m }} \exp\left\{\sum_{k=1}^m [u_k]_{i_k}  - \frac{C_{i_1 \ldots i_m }}{\gamma}  \right\}
    -\sum_{k = 1}^m u_{k}^Tp_k \Bigg],
\end{multline}
where $U = (u_1^T, \ldots, u_m^T)^T \in \R^{mn}$.

\end{proof}


\section{Algorithm Design Based on the Alternating Minimization Approach}\label{sec:algo_design}

In this section, we describe the proposed approach for designing an algorithm to approximately solve the MOT problem, based on an alternating minimization approach.

First, we introduce the tensor $B(U)\in \R_+^{n^m}$ with elements given as
\begin{multline*}
B_{i_1, \ldots, i_m}(u_1, \ldots, u_m)
    = \exp\left\{\sum_{k=1}^m [u_k]_{i_k}  - \frac{C_{i_1 \ldots i_m }}{\gamma}  \right\},
\end{multline*}
and element-wise sum given as
\begin{equation*}
    \Sigma(U) = \sum_{\substack{i_1, \dots, i_m \\ 1 \leq i_j \leq n,  1 \leq j \leq m }} B_{i_1, \ldots, i_m}(u_1, \ldots, u_m).
\end{equation*}
    
Moreover, it follows that the partial derivatives of the dual function $\phi $ are
\begin{multline}
    \label{part-deriv}
    \frac{1}{\gamma}\left[\frac{\partial \phi }{\partial u_\xi}\right]_\eta =  \sum_{\substack{i_1, \dots, i_m \\ 1 \leq i_j \leq n \\  1 \leq j \leq m \\i_\xi =\eta }} \frac{\exp\left\{\sum_{k=1}^m [u_k]_{i_k}  - \frac{C_{i_1 \ldots i_m }}{\gamma}  \right\}}{\Sigma(U)} -  [p_\xi]_{\eta} 
    \\
    =
    \frac{[p_\xi(B(U))]_\eta}{\Sigma(U)}- [p_\xi]_{\eta} .
\end{multline}

Therefore, as shown in the next lemma, we obtain a closed-form solution for alternating minimization of the dual problem.

\begin{lemma}
\label{Lm:block_min}
The iterations \[u_k^{t+1} \in \argmin_{u\in\mathbb{R}^n}\phi(u^t_1, \ldots, u^t_{k-1}, u, u^t_{k+1}, \ldots, u^t_m),\] can be written explicitly as
\[u^{t+1}_k=u^{t}_k+\ln p_k-\ln p_k(B(U^t)), \]
or entry-wise as
\begin{equation}
    [u^{t+1}_k]_\eta=[u^{t}_k]_\eta+\ln [p_k]_\eta-\ln [p_k(B(U^t))]_\eta \label{u-iter}.
\end{equation}
\end{lemma}
\begin{proof}
Consider the following tensor
\begin{multline*}
    B_{i_1, \ldots, i_m}(u^t_1, \ldots, u^t_{\xi-1}, u^{t+1}_\xi, u^t_{\xi+1}, \ldots, u^t_m)
    \\
    = \exp\left\{[u^{t+1}_\xi]_{i_\xi} + \sum_{k\neq\xi} [u^t_k]_{i_k}  - \frac{C_{i_1 \ldots i_m }}{\gamma}  \right\}
    \\
    =
    \frac{\exp[u^{t+1}_\xi]_{i_\xi}}{\exp[u^{t}_\xi]_{i_\xi}} \exp\left\{[u^{t}_\xi]_{i_\xi} + \sum_{k\neq\xi} [u^t_k]_{i_k}  - \frac{C_{i_1 \ldots i_m }}{\gamma}  \right\}
    \\
    = \frac{\exp[u^{t+1}_\xi]_{i_\xi}}{\exp[u^{t}_\xi]_{i_\xi}} B(U^t),
\end{multline*}
and plug in the expression \eqref{u-iter} from the lemma statement
\begin{multline*}
    \sum_{\substack{i_1, \dots, i_m \\ 1 \leq i_j \leq n \\  1 \leq j \leq m  }} B_{i_1, \ldots, i_m}(u^t_1, \ldots, u^t_{\xi-1}, u^{t+1}_\xi, u^t_{\xi+1}, \ldots, u^t_m)
    \\
    = 
    \sum_\eta \sum_{\substack{i_1, \dots, i_m \\ 1 \leq i_j \leq n \\  1 \leq j \leq m \\i_\xi =\eta }} B_{i_1, \ldots, i_m}(u^t_1, \ldots, u^t_{\xi-1}, u^{t+1}_\xi, u^t_{\xi+1}, \ldots, u^t_m)
    \\
    = \sum_\eta \frac{\exp[u^{t+1}_\xi]_{\eta}}{\exp[u^{t}_\xi]_{\eta}} [p_\xi(B(U^t))]_\eta
    \\
    \stackrel{\eqref{u-iter}}{=}\sum_\eta \frac{[p_\xi]_\eta}{[p_\xi(B(U^t))]_\eta} [p_\xi(B(U^t))]_\eta = 1.
\end{multline*}

Next, we plug \eqref{u-iter} in the optimality conditions
$ \frac{\partial \phi }{\partial [u_\xi]_\eta} =0 $
and show that the conditions are satisfied
\begin{multline*}
    [p_\xi]_{\eta}  =
    \\
    = \frac{\exp([u^{t+1}_\xi]_{\eta})  \sum_{\substack{i_1, \dots, i_m \\ 1 \leq i_j \leq n \\  1 \leq j \leq m \\i_\xi =\eta }} \exp\left\{\sum_{k\neq\xi} [u^t_k]_{i_k}  - \frac{C_{i_1 \ldots i_m }}{\gamma}  \right\}}
    {\sum_{\substack{i_1, \dots, i_m \\ 1 \leq i_j \leq n \\  1 \leq j \leq m  }} B_{i_1, \ldots, i_m}(u^t_1, \ldots, u^t_{\xi-1}, u^{t+1}_\xi, u^t_{\xi+1}, \ldots, u^t_m)} 
    \\
    = \exp([u^{t+1}_\xi]_{\eta})  \sum_{\substack{i_1, \dots, i_m \\ 1 \leq i_j \leq n \\  1 \leq j \leq m \\i_\xi =\eta }} \exp\left\{\sum_{k\neq\xi} [u^t_k]_{i_k}  - \frac{C_{i_1 \ldots i_m }}{\gamma}  \right\}
    \\
    =
    \frac{e^{[u^{t+1}_\xi]_{\eta}}} {e^{[u^{t}_\xi]_{\eta}}} \sum_{\substack{\ldots\\\ldots\\i_\xi =\eta }} B(U^t) 
    \stackrel{\eqref{u-iter}}{=} \frac{[p_\xi]_\eta}{[p_\xi(B(U^t))]_\eta}[p_\xi(B(U^t))]_\eta.
\end{multline*}
\end{proof}
Lemma \ref{Lm:block_min} implies that the dual objective $\phi$ can be explicitly minimized in each of the $m$ blocks of variables $u_k$, $k=1,...,m$, suggesting to use alternating minimization algorithms for the dual problem. Note that the nature of the Iterative Bregman Projections algorithm \cite{benamou2015iterative} is different since it is an alternating projection algorithm for the primal problem.

\subsection{General Primal-Dual Accelerated Alternating Minimization}
\label{sec:primal-dual}

In order to analyze the proposed algorithm, first we develop a general framework for primal-dual accelerated alternating minimization. We consider a general minimization problem \[
(P_1) \quad \quad \min_{x\in Q \subseteq E} \left\{ f(x) : \mathcal{A}x =b \right\},\] where $E$ is a finite-dimensional real vector space, $Q$ is a simple closed convex set, $\mathcal{A}$ is a given linear operator from $E$ to some finite-dimensional real vector space $H$, $b \in H$ is given. This problem template, in particular, covers Problem~\eqref{eq:EMOT_Pr_St}.
The Lagrange dual problem to Problem $(P_1)$ is
\[
(D_1) \quad \quad \max_{\lambda \in \Lambda} \left\{ - \la \lambda, b \ra  + \min_{x\in Q} \left( f(x) + \la \mathcal{A}^T \lambda  ,x \ra \right) \right\}.
\notag
\]
Here, we define $\Lambda=H^*$. Note also that Problem~\eqref{eq:dual_obj_lambda} is a particular case of this general dual template.
It is convenient to rewrite Problem $(D_1)$ in the equivalent form of a minimization problem
\begin{align}
& (P_2) \quad \min_{\lambda \in \Lambda} \left\{   \vp(\lambda)=\la \lambda, b \ra  + \max_{x\in Q} \left( -f(x) - \la \mathcal{A}^T \lambda  ,x \ra \right) \right\}. \notag
\end{align}
Since $f$ is convex, $\vp(\lambda)$ is a convex function. Thus, by Danskin's theorem  (see e.g. \cite{nesterov2005smooth}), its subgradient is
\begin{equation}
\nabla \vp(\lambda) = b - \mathcal{A} x (\lambda),
\label{eq:nvp}
\end{equation}
where $x (\lambda)$ is some solution of the convex problem
\begin{equation}
    \max_{x\in Q} \left( -f(x) - \la \mathcal{A}^T \lambda  ,x \ra \right).
    \label{eq:inner}
\end{equation}In what follows, we assume that $\vp(\lambda)$ is $L$-smooth and that the dual problem $(D_1)$ has a solution $\lambda^*$ and there exist some $R>0$ such that $\|\lambda^{*}\|_{2} \leqslant R$. We underline that the quantity $R$ will be used only in the convergence analysis, but not in the algorithm itself. 

To describe our algorithm we also need the following notation. The set $\{1,\ldots, N\}$ of indices of the orthonormal basis vectors $\{e_i\}_{i=1}^N$ is divided into $m$ disjoint subsets (blocks) $I_k$, $k\in\{1,\ldots,m\}$. Let $S_k(x)=x+\spn\{e_i:\ i\in I_k\}$, i.e. the affine subspace containing $x$ and all the points differing from $x$ only over the block $k$. 

The idea of the Algorithm \ref{PDAAM-2} is to use greedy alternating minimization steps in the dual and combine them with momentum, as in Nesterov's accelerated methods. This allows us to obtain an accelerated convergence rate for the dual problem. Further, we add a step which updates the primal variable, which is our actual objective, since it corresponds to the multimarginal transportation tensor.

\begin{algorithm}[ht]
\caption{Primal-Dual Accelerated Alternating Minimization (PD-AAM)}
\label{PDAAM-2}
\begin{algorithmic}[1]
   \STATE $A_0=\alpha_0=0$, $\eta^0=\zeta^0=\theta^0=\textbf{0}_{mn}$
   \FOR{$t \geqslant 0$}
    \STATE Set $\beta_t = \argmin\limits_{\beta\in [0,1]} \vp\left(\eta^t + \beta (\zeta^t - \eta^t)\right)$
    \STATE Set $\theta^t = \eta^t + \beta (\zeta^t - \eta^t)$
    \STATE Choose $i_t=\argmax\limits_{i\in\{1,\ldots,n\}} \|\nabla_i \vp(\theta^t)\|_2^2$
    \STATE Set $\eta^{t+1}=\argmin\limits_{\eta\in S_{i_t}(\theta^t)} \vp(\eta)$

    \STATE Find largest $a_{t+1}$ from the quadratic equation \[\vp(\theta^t)-\frac{a_{t+1}^2}{2(A_t+a_{t+1})}\|\nabla \vp(\theta^t)\|_2^2=\varphi(\eta^{t+1})\]\\
    \STATE  Set $A_{t+1} = A_{t} + a_{t+1}$
    \STATE Set $\zeta^{t+1} = \zeta^{t} - a_{t+1}\nabla \vp(\theta^t)$
    \STATE Set $\hat{x}^{t+1} = \frac{a_{t+1}x(\theta^{t})+A_t\hat{x}^{t}}{A_{t+1}},$ where $x(\theta^{t})$ is  the primal variable reconstruction (Eq.~\eqref{primal-var} in the case of MOT)
    \ENDFOR
    \ENSURE The points $\hat{x}^{t+1}$, $\eta^{t+1}$.
\end{algorithmic}
\end{algorithm}

The key result for this method is that it guarantees convergence in terms of the constraints and the duality gap for the primal problem, provided that the dual is smooth, in the spirit of \cite{dvurechensky2016primal-dual,dvurechensky2018decentralize,uribe2018distributed,guminov2019accelerated,dvurechensky2020stable,nesterov2020primal-dual}.

\begin{theorem}[\cite{2019arXiv190603622G}, Theorem 3]
\label{Th:PD_rate}
Let the objective $\vp$ in the problem $(P_2)$ be $L$-smooth and the solution of this problem be bounded, i.e. $\|\lambda^*\|_2 \leqslant R$. Then, for the sequences $\hat{x}_{t+1},\eta_{t+1}$, $t\geqslant 0$, generated by Algorithm \ref{PDAAM-2}, we have
\begin{align}
    f(\hat{x}^t) - f^* \leq f(\hat{x}^t) + \vp(\eta^t) 
    &\leqslant \frac{2mLR^2}{t^2}, \label{eq:gen_obj_bound}
    \\
    \|\mathcal{A} \hat{x}^t - b \|_2 &\leqslant \frac{8mLR}{t^2}. \label{eq:gen_constr_bound}
\end{align}
\end{theorem}
To apply this result we need to estimate the Lipschitz constant $L$ of the gradient of the dual objective and provide a bound $R$ for an optimal solution. 

Later, we will see the application of Theorem~\ref{Th:PD_rate} to the MOT problem based on the following change of variables.
\[x \rightleftharpoons X, \quad f(x) \rightleftharpoons F(X), \quad \varphi(\Lambda) \rightleftharpoons \phi(U) \rightleftharpoons \phi(\Lambda)\]
\[
    \{x:\mathcal{A}x=b\} \rightleftharpoons  \{X: p_k(X) = p_k, \quad \forall k \in \{1, \dots, m\}  \}
\]
\[
    Q \rightleftharpoons \{X \in \R_+^{n \times \ldots \times n} :\sum_{i_1, \dots, i_m} X_{i_1,\dots, i_m} = 1,\quad 1 \leq i_j \leq n\}
\]
The primal variable $X$ is reconstructed from the dual variable $U$ or $\Lambda$ using~\eqref{primal-var}.

\subsection{Bound for L}
\label{S:L_Bound}
We endow the space of transportation tensors with $1$-norm, which leads to 
the primal objective in \eqref{eq:EMOT_Pr_St} being strongly convex on the feasible set of this problem with parameter $\gamma$. Further, we use the $2$-norm for the dual space of Lagrange multipliers $\Lambda$ in \eqref{eq:dual_obj_lambda}. Hence, the dual objective in \eqref{eq:dual_obj_lambda} is $L$-smooth 
with the parameter $L \leq \|\mathcal{A}\|^2_{1 \rightarrow 2}/\gamma$ ~\cite{nesterov2005smooth}
. Here $\mathcal{A}:\R^{n^m}\rightarrow\R^{mn}$ is the linear operator defining the linear constraints of the problem, which, in the case of the multimarginal optimal transport problem, is defined by 
$\mathcal{A}\vectr{(X)} = (p_1(X)^T,\ldots, p_m(X)^T)^T$.
Thus, each column of the matrix $\mathcal{A}$ contains no more than $m$ non-zero elements, which are equal to one. Hence, since $\|\mathcal{A}\|_{1 \rightarrow 2}$ is equal to maximum $2$-norm of the column of this matrix, we have that $\|\mathcal{A}\|_{1 \rightarrow 2} = \sqrt{m}$. Finally, we have that $L \leq \frac{m}{\gamma}$.

\subsection{Bound for R}
We return to the particular dual problem \eqref{mot-dual} for the MOT problem to estimate the norm of an optimal dual solution in this particular case.
\begin{lemma}
For every $u^*_\xi$ entry of $U^* = ([u^*_1]^T, \dots, [u^*_m]^T)^T$ the following holds
\[
    \max_\eta [u_\xi^*]_\eta - \min_\eta [u_\xi^*]_\eta \leq -\ln \nu \min_\eta [p_\xi]_{\eta}. 
\]
\end{lemma}

\begin{proof}
By the optimality condition \eqref{part-deriv}
\begin{multline*}
    0 = \frac{\partial \phi }{\partial [u_\xi]_\eta}
    = - [p_\xi]_{\eta} 
    \\
    + \frac{\exp([u_\xi]_{\eta})}{ \Sigma(U)} \sum_{\substack{i_1, \dots, i_m \\ 1 \leq i_j \leq n \\  1 \leq j \leq m \\i_\xi =\eta }} \exp\left\{\sum_{k\neq\xi} [u_k]_{i_k}  - \frac{C_{i_1 \ldots i_m }}{\gamma}  \right\},
\end{multline*}
where $\nu = \exp\frac{-\|C\|_\infty}{\gamma}$.
Since $p_\xi \in \Delta_n$, we obtain the bound for the the solution of the above optimality conditions

\begin{multline}
    1\geq [p_\xi]{_\eta} 
    \\
    = \frac{\exp([u^*_\xi]_\eta)}{ \Sigma(U^*)} \sum_{\substack{i_1, \dots, i_m \\ 1 \leq i_j \leq n \\  0 \leq j \leq m \\i_\xi =\eta }} \exp\left\{\sum_{k\neq\xi} [u^*_k]_{i_k}  - \frac{C_{i_1 \ldots i_m }}{\gamma}  \right\} 
    \\
    \geq \nu \exp([u^*_\xi]_\eta) \Sigma(U^*)^{-1} \sum_{\substack{i_1, \dots, i_m \\ 1 \leq i_j \leq n \\  0 \leq j \leq m \\i_\xi =\eta }} \exp\left\{\sum_{k\neq\xi} [u_k^*]_{i_k}   \right\}
    \\
    =
    \nu \exp([u^*_\xi]_\eta) \Sigma(U^*)^{-1} \sum_{\substack{{k=1}\\k\neq\xi}}^m \la \one, e^{u_k^*}\ra.
\end{multline}
From the above inequality we have
\begin{equation}
    [u_\xi^*]_\eta \leq \ln\Sigma(U^*)  - \ln \nu - \ln \sum_{\substack{{k=1}\\k\neq\xi}}^m \la \one, e^{u_k^*}\ra.
    \label{u-upper}
\end{equation}
On the other hand, 
\begin{multline}
    [p_\xi]{_\eta} = \frac{\exp([u_\xi^*]_\eta)}{ \Sigma(U^*)} \sum_{\substack{i_1, \dots, i_m \\ 1 \leq i_j \leq n \\  0 \leq j \leq m \\i_\xi =\eta }} \exp\left\{\sum_{k\neq\xi} [u_k^*]_{i_k}  - \frac{C_{i_1 \ldots i_m }}{\gamma}  \right\} 
    \\ 
    \leq \exp([u_\xi^*]_\eta) \Sigma(U^*)^{-1} \sum_{\substack{i_1, \dots, i_m \\ 1 \leq i_j \leq n \\  0 \leq j \leq m \\i_\xi =\eta }} \exp\left\{\sum_{k\neq\xi} [u_k^*]_{i_k}   \right\},
\end{multline}
leads to
\begin{equation}
    [u_\xi^*]_\eta \geq \ln [p_\xi]{_\eta} + \ln\Sigma(U^*)  - \ln \sum_{\substack{{k=1}\\k\neq\xi}}^m \la \one, e^{u_k^*}\ra.
    \label{u-lower}
\end{equation}
Combining \eqref{u-lower} and \eqref{u-upper} we have, for all $\xi = 1,...,m$,
\[
    \max_\eta [u_\xi^*]_\eta - \min_\eta [u_\xi^*]_\eta \leq -\ln \nu \min_\eta [p_\xi]_{\eta}. 
\]
\end{proof}

\begin{lemma}
Defining $\Lambda^0 = -\frac{\gamma}{m}\one_{mn}$, there exists a solution $\Lambda^*$ of the dual problem \eqref{eq:dual_obj_lambda} such that
\begin{align*}
    R=\|\Lambda^*-\Lambda^0\|_2 \leqslant\frac{\sqrt{mn}}{2}\left(\|C\|_\infty-\frac{\gamma}{2}\ln{\min\limits_{i,j}\{[p_i]_j\}}\right).
\end{align*}
\end{lemma}

\begin{proof}
We begin by deriving an upper bound on $\|(u_1^{*T}, \ldots, u_m^{*T})^T\|_2$. Using the results of the previous lemma, it remains to notice that the objective $\phi(U)$ is invariant under transformations $u_i\to u_i+t_i\mathbf{1}$, 
$t_i\in\mathbb{R}$ for $i \in \{1, \ldots, m\}$, so there must exist some solution with
$\max_\eta [u^*_i]_\eta = - \min_\eta [u^*_i]_\eta = \|u_i^*\|_\infty$, so
\begin{align*}
    \|u_i^*\|_\infty&\leqslant-\frac{1}{2}\ln\nu\min_\eta [p_i]_\eta .
\end{align*}
As a consequence,
\begin{align*}
    \|u^*_i\|_2&\leqslant\sqrt{n}\|U^*\|_\infty\leqslant
    \\
    &\leqslant-\frac{\sqrt{n}}{2}\ln\nu\min_{i,j}\{[p_i]_{j}\}
    \\
    &\leqslant\frac{\sqrt{n}}{2}\left(\frac{\|C\|_\infty}{\gamma}-\frac{1}{2}\ln{\min\limits_{i,j}\{[p_i]_{j}\}}\right).
\end{align*}
and 
\begin{multline*}
    \|U^*\|_2 = \sqrt{ \sum_i^m \|u^*_i\|_2^2} \leq \frac{\sqrt{mn}}{2}\left(\frac{\|C\|_\infty}{\gamma}-\frac{1}{2}\ln{\min\limits_{i,j}\{[p_i]_{j}\}}\right)
\end{multline*}

By definition, $u_i=-\frac{1}{\gamma}\lambda_i-\frac{1}{m}\mathbf{1}$, so we have the inverse transformation $\lambda_i=-\gamma u_i -\frac{\gamma}{m}\mathbf{1}$. Finally, with $\Lambda^0 = -\frac{\gamma}{m}\one_{mn}$
\begin{multline*}
    R=\|\Lambda^* - \Lambda^0\|_2=
    \\
    =\Big\Vert(-\gamma u_1^* -\frac{\gamma}{m}\mathbf{1}, \dots, -\gamma u_m^* -\frac{\gamma}{m}\mathbf{1})
    \\
    -(-\frac{\gamma}{m}\mathbf{1}, \dots, -\frac{\gamma}{m}\mathbf{1})\Big\Vert_2 =\|-\gamma(u_1^*, \dots, u_m^*)\|_2
    \\
    =\gamma\|U^*\|_2 \leq \frac{\sqrt{mn}}{2}\left(\|C\|_\infty-\frac{\gamma}{2}\ln{\min\limits_{i,j}\{[p_i]_j\}}\right).
\end{multline*}

\end{proof}

\subsection{Projection on the feasible set}
The Algorithm  \ref{PDAAM-2} may return a point in the primal space which does not satisfy the equality constraints. In this subsection, we provide a procedure to project approximate transport tensor to obtain a feasible point for the primal problem, i.e.
find such $\widehat X \approx \hat X^t$ that $p_i(\widehat X) = p_i$. To do this we formulate Algorithm \ref{alg:round}, which is a generalization of rounding procedure in \cite{altschuler2017near-linear}, see also \cite{2019arXiv191000152L}.

\begin{algorithm}[ht]
\caption{Multimarginal Rounding}
\label{alg:round}
\begin{algorithmic}[1]
    \STATE $V_1 = U$
   \FOR{$r =1,\cdots,m-1$}
    \STATE $[{X}_r]_i  = \min \left\{[p_r ]_i/[p_r(V_r)]_i,1\right\}$
    \STATE $X_r = \text{DiagTensor}({x}_r)$
    \STATE $V_{r+1}  = \text{ProdTensor}_r(V_r ,X_r)$
    \ENDFOR
    \FOR{$r =1,\cdots,m$}
    \STATE $\text{err}_r = p_r - p_r(V_m)$
    \ENDFOR
    \ENSURE $\hat{V} = V_m  +  \bigotimes_{r=1}^m {\text{err}}_r / \| {\text{err}}_m \|_1^{m-1}$
\end{algorithmic}
\end{algorithm}

Note that in Algorithm~\ref{alg:round} the function $\text{DiagTensor}(\cdot)$ takes a vector as input and outputs a $m$-dimensional tensor with the input as its diagonal. Moreover,  $\text{ProdTensor}_r(A,B)$ takes two $m$-dimensional tensors as input, and multiplies them in the direction $r$. We use $\bigotimes$ to denote the tensor product of the input factors. The next lemma shows that the output of Algorithm~\ref{alg:round} is in the desired space with the corresponding marginals, and bounds the error induced by the projection.

\begin{lemma}
\label{Lm:Rounding}
Let $\{p_k\}_{k=1}^m \in \Delta_n$, and $U\in \mathbb{R}_+^{n\times \dots \times n}$, then Algorithm~\ref{alg:round} outputs a matrix $\hat{F}$ with marginals $\{p_k\}_{k=1}^m$, satisfying
$
    \|U - \hat{V}\|_1 \leq 2 \sum_{r=1}^m\|p_r - p_r(U ) \|_1.
$
\end{lemma}

\begin{proof}
Initially, note that for all $k=1,\dots,m$, we have
\begin{align*}
    p_k(\hat{V}) & = p_k(V_m) + p_k(\bigotimes_{r=1}^m \text{err}_r / \| \text{err}_m \|_1^{m-1}) \\
    & = p_k(V_m) + \text{err}_k  = p_k.
\end{align*}
Thus, the output of the $\hat{U}$ has the desired marginals. Now, define $I = \|U\|_1 - \|V_m\|_1$, thus,
\begin{align*}
    I = \sum_{r=1}^m \sum_{i=1}^n ([p_r(V_r)]_i - [p_r]_i)_+
\end{align*}
Moreover, we have
\begin{align*}
    &\|\hat{V} - U \| \leq I + \|\bigotimes_{r=1}^m \text{err}_r \|/ \| \text{err}_m \|_1^{m-1}=\\ 
    &  I + 1 - \|V_m\|_1  = 2I + 1 - \|U\|_1  = 2\sum_{r=1}^m\|p_r - p_r(U ) \|_1,
\end{align*}
where the last line follows the same arguments as the proof of Lemma~$7$ in~\cite{altschuler2017near-linear}.

\end{proof}


\section{Complexity of Multimarginal OT
}\label{sec:complexity}
In this section, we prove the computational complexity of finding a $\eps$-solution for the original \textit{non-regularized} MOT problem~\eqref{prob:multi_OT}, i.e. we estimate the complexity to find $\widehat{X}$ satisfying all the constraints in \eqref{prob:multi_OT} and also satisfying
\begin{equation}
    \la  C,\widehat X  \ra \leq \la C,X^*\ra +\eps,
\end{equation}
where $X^*$ is an optimal solution for \eqref{prob:multi_OT}.The approximation is produced by Algorithm \ref{Alg:MOTbyAAM} below.

To obtain its complexity, we combine all the above building blocks, i.e., analysis of the PD-AAM algorithm and estimates for $R$ and $L$, and the rounding procedure.

\begin{algorithm}[tb]
  \caption{Approximate MOT by PD-AAM}
  \label{Alg:MOTbyAAM}
  \begin{algorithmic}[1]
    \REQUIRE{Accuracy $\e$.}
    \STATE Set $\gamma = \frac{\e}{2 m\ln n}$, $\e' = \frac{\e}{8 \|C\|_\infty}$. 
    \STATE Define $\tilde{p}_k=\left(1 - \frac{\e'}{4m}\right)p_k +  \frac{\e'}{4mn} \one_n $, $k=1,...,m$. 
    \STATE Apply PD-AAM to the dual problem \eqref{mot-dual} with marginals $\tilde{p}_k$, $k=1,...,m$ until the stopping criterion 
    $2\sum_{k=1}^m \|p_k(\widehat{X}^t)- \tilde{p}_k\|_1 + F(\hat{X}^t) + \phi(\eta^t) \leq \e/2$.
    \STATE Find $\widehat{X}$ as the projection of $\widehat{X}^t$ on $\{  X \in \R_+^{n \times \ldots \times n}, \;\; 
          p_k(X) = p_k, \ \forall k = 1, \dots, m \}$ by the Algorithm \ref{alg:round}.
    \ENSURE $\widehat{X}$.
  \end{algorithmic}
\end{algorithm}

To adapt Algorithm \ref{PDAAM-2} to Problem \eqref{mot-dual}, one should replace \textit{Step 4} with:
Choose $I=\argmax\limits_{i\in\{1,\ldots,m\}} \left \|\frac{\partial \phi }{\partial u_i}(\theta^t)\right \|_2$, 
and \textit{Step 5} with: Set 
            \[\eta^{t+1}_{i} = \Big \{ \begin{array}{ll}
                \theta^{t}_{i}+\ln p_{i}-\ln p_{i}(B(\theta^t)), & i = I \\ 
                \theta^t_i, & \text{otherwise}.
            \end{array} \]

\begin{theorem}
The output $\hat{X}$ of Algorithm~\ref{Alg:MOTbyAAM} is an $\eps$-solution for the original \textit{non-regularized} MOT problem~\eqref{prob:multi_OT}, e.g.
\begin{equation}
\label{eq:non_reg_sol_def}
    \la  C,\widehat X  \ra \leq \la C,X^*\ra +\eps.
\end{equation}
\end{theorem}

\begin{proof}
By Lemma \ref{Lm:Rounding}, $\widehat{X}$ is a feasible point for Problem~\eqref{prob:multi_OT}. 
Let us estimate the objective residual. We have
\begin{multline}
\label{eq:Th:OTbySinkhComplPr1}
    \la C, \widehat{X} \ra =  \la C, X^* \ra + \la C, X^*_\gamma - X^* \ra + \la C, \widehat{X}^t - X^*_\gamma \ra 
    \\ + \la C, \widehat{X} - \widehat{X}^t \ra 
    \le \la C, X^* \ra + \gamma m \ln n + F(\hat{X}^t) + \phi(\eta^t) 
    \\
     + 2 \sum_{k=1}^m \left\|p_k(\widehat{X}^t)- {p}_k\right\|_1 \|C\|_\infty,
\end{multline} 
where $\widehat{X}^t$ is the output of Algorithm \ref{PDAAM-2}, $\widehat X$ is a projection of $\widehat{X}^t$ by Algorithm \ref{alg:round} on the feasible set, $X^*$ is a solution to the non-regularized multimarginal OT problem \eqref{prob:multi_OT}, $X_\gamma^*$ is a solution to the entropy-regularized multimarginal OT problem \eqref{eq:EMOT_Pr_St}.
To obtain the last inequality we used the fact that the Entropy on the standard simplex in the dimension $n^m$ belongs to the interval $-H(X) \in [-m\ln n,0]$, and, hence, $\la C, X^*_\gamma - X^* \ra \leq 0$ and
\begin{align}
\label{eq:OTbyGDPr3}
\la C, \widehat{X}^t &- X_{\gamma}^* \ra =  (\la C, \widehat{X}^t \ra - \gamma H(\widehat{X}^t)) \notag \\
 &- (\la C, X_{\gamma}^* \ra - \gamma H(X_{\gamma}^*)) + \gamma (H(\widehat{X}^t)-H(X_{\gamma}^*)) \notag \\
& \stackrel{\eqref{eq:gen_obj_bound}}{\leq} F(\hat{X}^t) + \phi(\eta^t) +  \gamma m \ln n.
\end{align}
Finally, by the H\"older inequality and Lemma \ref{Lm:Rounding}, 
\small
\begin{equation*}
    \la C, \widehat{X} - \widehat{X}^t \ra \leq \|C\|_\infty \|\widehat{X} - \widehat{X}^t\|_1 
    \leq 2\|C\|_\infty \sum_{k=1}^m \|p_k(\widehat{X}^t) - p_k\|_1.
\end{equation*}
\normalsize
This finishes the proof of inequality \eqref{eq:Th:OTbySinkhComplPr1}.

Further, we have 
\small
\begin{equation*}
    \sum_{k=1}^m \|p_k(\widehat{X}^t)- {p}_k\|_1 \le \sum_{k=1}^m \left( \left\|p_k(\widehat{X}^t)- \tilde{p}_k\right\|_1+ \|\tilde{p}_k - p_k\|_1 \right) \leq \e',
\end{equation*}
\normalsize
by the construction of $\tilde{p}_k$ and the stopping criterion in step 3 of Algorithm \ref{Alg:MOTbyAAM}. Combining this, \eqref{eq:Th:OTbySinkhComplPr1}, the choice of~$\gamma$ and~$\e'$ as well as the stopping criterion in step 3 of Algorithm \ref{Alg:MOTbyAAM}, we obtain that \eqref{eq:non_reg_sol_def} holds.
\end{proof}
It remains to estimate the complexity of the algorithm. By Theorem \ref{Th:PD_rate}, we obtain that
\small
\begin{multline*}
\sum_{k=1}^m \left\|p_k(\widehat{X}^t)- \tilde{p}_t\right\|_1 \leq \sqrt{mn}\|\mathcal{A} \widehat{X}^t - b \|_2 \leqslant \frac{8m^\frac{3}{2}n^\frac{1}{2}LR}{t^2}
\\
\leq \frac{8m^\frac{3}{2}n^\frac{1}{2}}{t^2}\cdot \frac{m \cdot 2m \ln n}{\e} \cdot \frac{\sqrt{mn}\left(\|C\|_\infty+\frac{\e}{4m \ln n}\ln{\frac{4 m  n \cdot 8 \|C\|_{\infty}}{\e}}\right) }{2}
\\
= \frac{8m^{4}n \|C\|_{\infty}\ln n}{\e t^2}\left(1+\frac{\e}{4m \|C\|_{\infty} \ln n}\ln{\frac{32 m  n \|C\|_{\infty}}{\e}}\right),
\end{multline*}
\normalsize
where the operator $\mathcal{A}$ is defined in Sect \ref{S:L_Bound} and we used that by the choice of $\tilde{p}_k$, $\min\limits_{i,j}\{[p_i]_j\} \geq \frac{\e'}{4mn}$.
At the same time,
\small
\begin{multline*}
F(\hat{X}^t) + \phi(\eta^t) \leqslant \frac{2mLR^2}{t^2}  \\
\leq \frac{2m}{t^2} \cdot \frac{m \cdot 2m \ln n}{\e} \cdot \frac{mn}{4}\left(\|C\|_\infty+\frac{\e}{4m \ln n}\ln{\frac{32 m  n  \|C\|_{\infty}}{\e}}\right)^2 \\
= \frac{m^4 n \|C\|_\infty ^2\ln n}{t^2 \e} \left(1+\frac{\e}{4m \|C\|_\infty \ln n}\ln{\frac{32 m  n  \|C\|_{\infty}}{\e}}\right)^2.
\end{multline*}
\normalsize
Let us denote $\delta_\e = 1+\frac{\e}{4m \|C\|_\infty \ln n}\ln{\frac{32 m  n  \|C\|_{\infty}}{\e}}$. Since $\e$ is small and $m,n$ are large, we can think of this quantity as $\delta_\e = O(1)$. Then, to satisfy the stopping criterion in step 3 of Algorithm \ref{Alg:MOTbyAAM} we need to take 
\begin{multline*}
t \geq \sqrt{\frac{128m^{4}n \|C\|_{\infty}^2\delta_\e\ln n }{\e^2}} = \widetilde{O}\left( \frac{m^{2}n^{1/2} \|C\|_{\infty} }{\e}\right), \text{and}
\end{multline*}
\begin{multline*}
t \geq \sqrt{\frac{4m^{4}n \|C\|_{\infty}^2\delta_\e^2\ln n }{\e^2}} = \widetilde{O}\left( \frac{m^{2}n^{1/2} \|C\|_{\infty} }{\e}\right).
\end{multline*}
Since in each iteration we need to calculate the full gradient of the dual objective, which amounts to calculating $m$ marginals $p_k(B(U))$, $k=1,...,m$ of the $m$-dimensional tensor $B(U)$, the cost of this operation is   $O(mn^{m})$ and it dominates the complexity of other operations in each iteration. This gives the following theorem and the main result of the paper.
\begin{theorem}
The computational complexity of finding an $\varepsilon$-approximate solution for the non-regularized MOT problem using Algorithm \ref{Alg:MOTbyAAM} is
\[
\widetilde{O} \left(\frac{m^3n^{m+1/2}\|C\|_{\infty}}{\varepsilon} \right).
\]
\end{theorem}


We now discuss the scalability of the proposed algorithm. As already mentioned, the most expensive operation on each iteration is the calculation of $m$ marginals $p_k(B(U))$ of the $m$-dimensional tensor $B(U)$. This operation can be organized in parallel if we store this tensor in shared memory and allow $m$ workers to access it. Then, they can independently calculate all the marginals. The total amount of arithmetic operations remains the same, but the work time is now proportional to $n^m$ rather than $mn^m$.  

Next, we compare our complexity results with the estimates in the preprint \cite{2019arXiv191000152L}. By inspecting their Algorithm 2 and Algorithm 5, we see that similarly to our algorithm, in each iteration, they need to calculate all the marginals  (which they denote by $r_i(B(\beta))$) to choose the block $I$, which will be updated. The complexity of this operation dominates the complexity of other operations in each step. Thus, since each iteration in their algorithms and our algorithm is the same, we compare the iteration complexity of the algorithms. The iteration complexity of our algorithm is $\widetilde{O}\left(  {m^{2}n^{1/2} \|C\|_{\infty} }/{\e}\right)$. The iteration complexity of the multimarginal Sinkhorn's algorithm \cite{2019arXiv191000152L} is $\widetilde{O}\left(  {m^{3} \|C\|_{\infty}^2 }/{\e^2}\right)$, which has worse dependence on $\e$ and $m$ than our bound. The claimed iteration complexity of multimarginal RANDKHORN algorithm in \cite{2019arXiv191000152L} is $\widetilde{O}\left(  {m^{8/3}n^{1/3} \|C\|_{\infty}^{4/3} }/{\e}\right)$, which has worse dependence on $m$ and $ \|C\|_{\infty}$ than our bound. Moreover, the multimarginal RANDKHORN is a randomized algorithm, and its complexity is estimated on average, whereas our algorithm and complexity are deterministic.



\section{Experiments}\label{sec:experiments}
This section provides a numerical comparison of multimarginal Sinkhorn's algorithm from \cite{2019arXiv191000152L} with our AAM method. We performed experiments using randomly chosen vectors $p_i \in \Delta_n$ and tensor $C \in \R_+^{n^m}$. We slightly modified the smaller values of $p_i$ as described above to lower bound their minimal value. We choose several values of accuracy $\eps \in [0.25, 0.0125]$, and run the methods until the stopping criterion was reached. One can see that our AAM algorithm outperforms multimarginal Sinkhorn's algorithm from \cite{2019arXiv191000152L}. \footnote{ The code available \url{https://rb.gy/siirke}}. Unfortunately, we were not able to implement the multimarginal RANDKHORN algorithm since its stopping criterion $\bar{E}_t > \varepsilon'$ depends on \textit{expected} residual in the constraints given in \cite[Eq. (28)]{2019arXiv191000152L}, which is unavailable in practice.

\begin{figure}[ht]
\begin{center}
\centerline{\includegraphics[width=0.97\columnwidth]{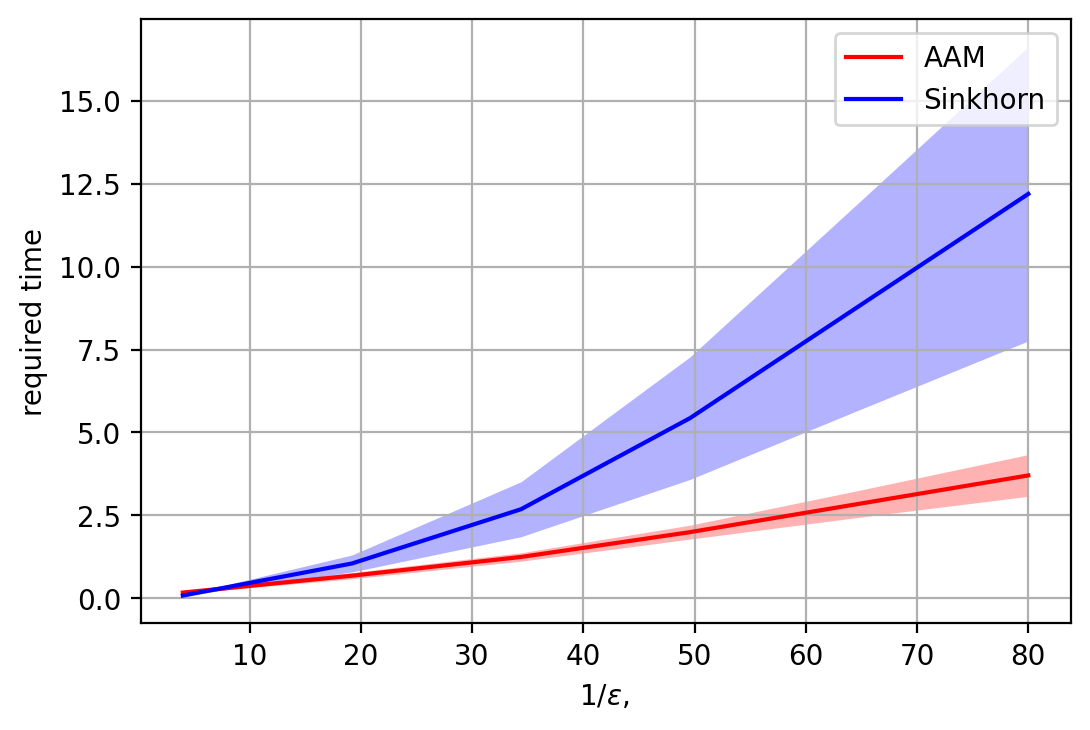}}
\vspace{-0.4cm}
\caption{Performance comparison between multimarginal Sinkhorn's algorithm and Algorithm \ref{Alg:MOTbyAAM} ($n= 15$, $m=4$).}
\label{exper-1}
\end{center}
\end{figure}

\begin{figure}[ht]
\begin{center}
\centerline{\includegraphics[width=0.97\columnwidth]{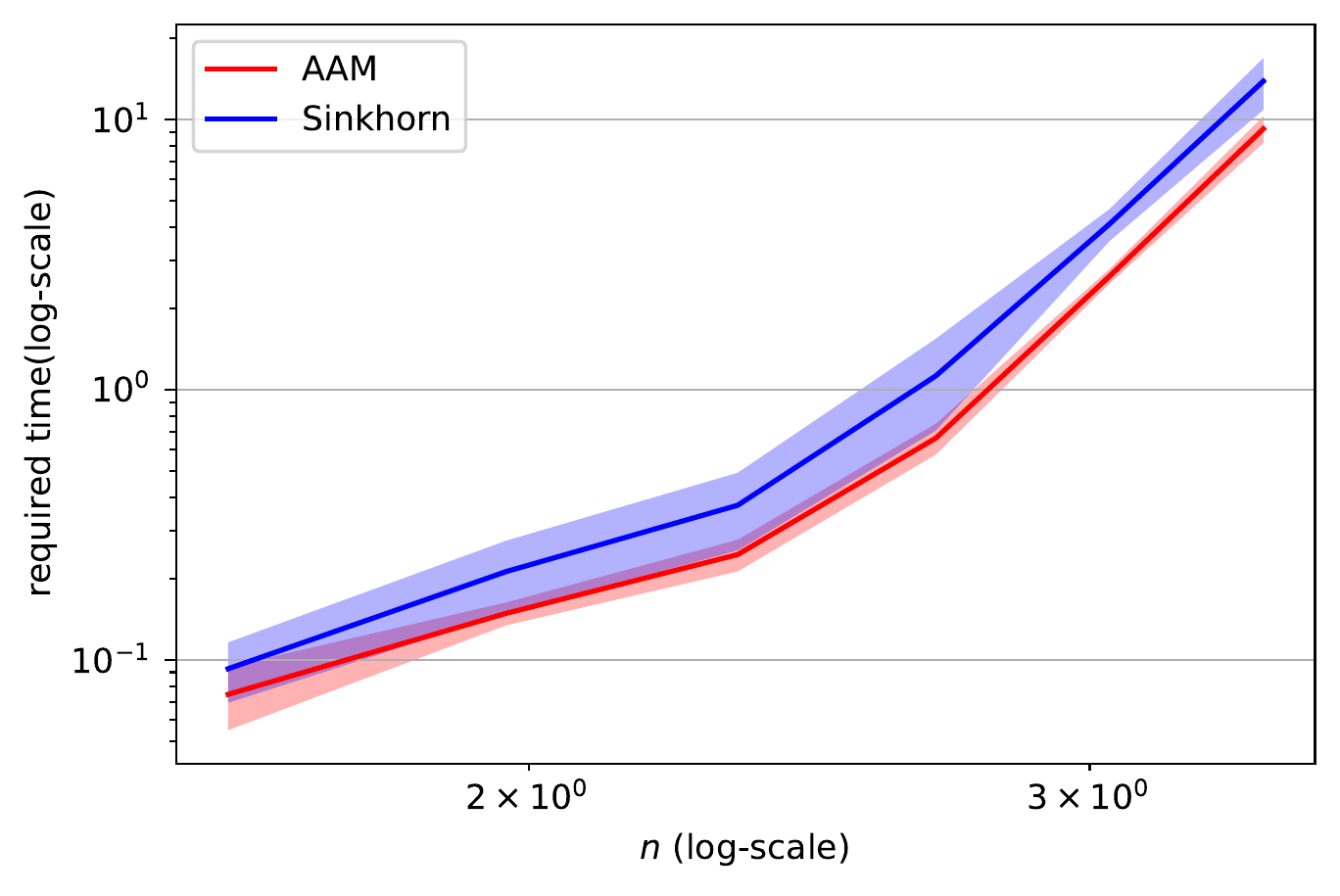}}
\vspace{-0.4cm}
\caption{Performance comparison between multimarginal Sinkhorn's algorithm and Algorithm \ref{Alg:MOTbyAAM} ($m=4$, $\eps=0.05$).}
\label{exper-2}
\end{center}
\vspace{-0.2cm}
\end{figure}

\vspace{-0.2cm}

\section{Conclusions}\label{sec:conclusions}

We provide a novel algorithm for the computation of approximate solutions to the multimarginal optimal transport problem. Our results are based on a new primal-dual analysis of the entropy regularized optimal transport problem. We show that the iteration complexity of our algorithm is better than the state-of-the-art methods in a large set of problem regimes to the number of distributions, dimension of the distributions, and desired accuracy.

As a byproduct of our analysis, given that the Wasserstein barycenter of a set of distributions can be recovered from the optimal multimarginal transport plan~\cite{benamou2015iterative}, we provide some evidence of an exponential complexity bound for the computation of the free-support barycenter which is known to be a non-convex problem. 

Future work will include the study of fully decentralized approaches and extensive experimental results for applications related to signal processing.

\vspace{-0.0cm}



\bibliography{references}
\bibliographystyle{IEEEtran}

\end{document}